\documentclass[reqno]{amsart}
\usepackage{amssymb,amsmath,amsfonts,amscd,amsthm,pb-diagram}
\usepackage{amsbsy,bm,color}
\newtheorem{theorem}{Theorem}[section]
\newtheorem{definition}{Definiton}[section]

\newtheorem{proposition}{Proposition}[section]
\newtheorem{corollary}{Corollary}[section]
\theoremstyle{remark}
\newtheorem{remark}{Remark}[section]

\begin{document}
\title{The first nonzero eigenvalue of the weighted $p$-Laplacian on differential forms}

\author{Mingzhu Miao, Xuerong Qi$^{*}$ and Jiabin Yin}

\address{School of Mathematics and Statistics, Zhengzhou University, Henan 450001, China}
\email{xrqi@zzu.edu.cn}
\address{School of Mathematics and Statistics, Xinyang Normal University, Henan 464000, China}
\email{jiabinyin@126.com}
	
\thanks{The second author was supported by the NSF of China (Grant No.11401537) and the NSF of Henan Province (Grant No. 252300421476).
The third author was supported by NSF of China (Grant No.12201138) and Mathematics Tianyuan fund project (Grant No.12226350).}
\thanks{$^{*}$Corresponding author.}

\date{}


\begin{abstract}
We introduce the weighted $p$-Laplace operator acting on differential forms on a metric measure space, which is a natural generalization of the $p$-Laplace operator defined by Seto \cite{Seto}. We obtain some sharp lower bounds of the first nonzero eigenvalue for the weighted $p$-Laplacian. Our results extend
an estimate of Seto \cite{Seto}, as well as the eigenvalue estimates derived by Cui-Sun \cite{Cui-Sun} for closed submanifolds.
\end{abstract}
\maketitle

\numberwithin{equation}{section}
\section{Introduction}\label{sect:1}

Let $(M,g)$ be an $n$-dimensional compact Riemannian manifold. For each integer $0 \leq k \leq n$, the Hodge Laplacian acting on $k$-forms of $M$ is defined by
$$\Delta\omega=d\delta\omega+\delta{d}\omega,\ \ \ \forall \ \omega \in  \Omega^{k}(M),$$
where $d$ and $\delta$ are the exterior differential and co-differential operators on $\Omega^{k}(M)$, respectively.
Eigenvalue problems of the Hodge Laplacian provide key insights into
the relationship between geometry and topology of Riemannian manifolds.
Therefore, it is very important to obtain characterizations of the eigenvalues of the Hodge Laplacian in terms of geometric invariants, particularly curvature (see e.g. \cite{Chavel,Kwong,Li1,Matei,Raulot-Savo,Savo,Savo1, Schoen-Yau}).
For manifolds without boundary, we point out the work of Gallot and Meyer \cite{Gallot-Meyer,Gallot-Meyer1} who establish the following estimate for the first eigenvalue by utilizing lower bounds of the curvature operator:
\begin{theorem}\label{thm:1.1}{\rm (\cite{Gallot-Meyer,Gallot-Meyer1})}
Let $M$ be an $n$-dimensional closed Riemannian manifold with
the curvature operator bounded from below by $c\in\mathbb{R}$ and $1\leq k \leq\frac{n}{2}$. Then the first nonzero eigenvalue $\lambda_{1,k}$ of
the Hodge Laplacian $\Delta$ satifies
\begin{equation}\label{1.1}	
  \lambda_{1,k}\geq k(n-k+1)c.
\end{equation}
The equality holds for the $n$-sphere with the usual metric.
\end{theorem}

In the past two decades, increasing attention has been focused on eigenvalue
problems for nonlinear operators.
Motivated from the variational characterization of the Laplacian eigenvalue problem, Seto \cite{Seto} computed the Euler-Lagrange equation of the $L^p$-Dirichlet integral on $k$-forms introduced in \cite{Scott} by:
\begin{equation}\label{eqn:1.2}
  \mathcal{F}[\omega]:=\int_{M}\big(|d\omega|^p+|\delta\omega|^p\big), \ \ \ \forall \ \omega\in\Omega^k(M),
\end{equation}
and derived the definition of the $p$-Laplacian acting on $k$-forms:
$$\Delta_{p}\omega:=d\big(|\delta\omega|^{p-2}\delta\omega\big)
+\delta\big(|d\omega|^{p-2}d\omega\big), \ \ \ \forall \ \omega\in\Omega^{k}(M).$$
When $p=2$, this becomes the usual Hodge Laplacian. For $p\neq 2$ and $k=0$, $\Delta_{p}$ becomes the usual $p$-Laplacian.
The corresponding eigenvalue equation is given by
\begin{equation}
\Delta_{p}\omega=\lambda|\omega|^{p-2}\omega,  \ \ \ \omega\in\Omega^{k}(M).
\end{equation}
Eigenvalue estimates for the $p$-Laplacian acting on functions are intensively studied in a huge literature (e.g. \cite{Lindqvist,Matei,Naber-Valtorta,Seto-Wei}). Regarding eigenvalue estimates for the $p$-Laplacian acting on differential forms, Seto \cite{Seto} proved the following result:
\begin{theorem}\label{thm:1.2}{\rm (\cite{Seto})}
Let $M$ be an $n$-dimensional closed Riemannian manifold with
the curvature operator bounded from below by $c\in \mathbb R$ and $p\geq 2$. Then the first nonzero eigenvalue $\lambda_{1,k}$ of the $p$-Laplacian $\Delta_{p}$ acting on $k$-forms of $M$ satisfies
\begin{equation}\label{1.3}	
  \lambda_{1,k}\geq\left(\frac{k(n-k)c}{2^{1-\frac{2}{p}}
  \big(C+\frac{p-2}{2}\big)}\right)^{\frac{p}{2}},
\end{equation}
where $$C={\rm max}\left\{\frac{k}{k+1},\frac{n-k}{n-k+1}\right\}.$$
\end{theorem}

A natural problem is whether the above conclusion can be extended to metric measure spaces?

Now recall that a compact metric measure space $(M,g,e^{-f}dv)$ is a compact Riemannian manifold $(M,g)$ with the weighted measure $e^{-f}dv$, where $dv$ is the Riemannian volume element related to the metric $g$, $f$ is a smooth real-valued function on $M$.
A metric measure space is also called a weighted manifold or a Bakry-\'{E}mery manifold.
The initial motivation for studying such kind of manifolds was to
 model diffusion processes \cite{Bakry-Emery}. However, they have by now become famous in the study of self-similar solutions of the Ricci flow, the so-called Ricci solitons. Also they appear in the analysis of shrinkers, which represent a special class of solutions of the mean curvature flow. We refer to \cite{Colding-Minicozzi} and \cite{Impera-Rimoldi-Savo} for more details.

Denote by $\nabla$, $\Delta$ and $\nabla^2$ the gradient, the Hodge Laplacian and the Hessian on $M$, respectively.
The adjoint of the exterior derivative $d$ with respect to the measure $d\mu=e^{-f}dv$ is given by
\begin{equation}\label{eqn:1.4}	
  \delta_f:=\delta+\iota_{\nabla f},
\end{equation}
where $\iota_{\nabla\mathnormal{f}}$ denotes the interior product associated with the gradient vector field $\nabla f$.
Define the weighted Hodge Laplacian $\Delta_f$ as follows:
\begin{equation}\label{eqn:1.5}	
  \Delta_f:=\delta_f d+d\delta_f: \ \Omega^k(M) \to \Omega^k(M).
\end{equation}
For compact Riemannian manifolds without boundary or with convex boundary, Bakry-Qian \cite{Bakry-Qian} established a uniform lower bound estimate for the first nonzero eigenvalue of the weighted Laplacian acting on functions. For the case of compact Riemannian manifolds with boundary, Futaki-Li-Li \cite{Futaki-Li} proved a lower bound for the first nonzero eigenvalue of the weighted Laplacian;
Branding-Habib \cite{Branding-Habib} derived various eigenvalue estimates for the weighted Hodge Laplacian acting on differential forms.

Motivated by Seto's work \cite{Seto} on the $p$-Laplacian, we consider the $L^p$-Dirchlet integral for $k$-forms on a closed metric measure space $(M,g,d\mu)$ with respect to the weighted measure $d\mu$ as follows:
\begin{equation}\label{eqn:1.6}
  \mathcal{F}_{f}[\omega]
  :=\int_{M}\big(|d\omega|^p+|\delta_{f}\omega|^p\big)d\mu, \ \ \ \omega \in\Omega^k(M).
\end{equation}
Denote the space of the weighted harmonic $k$-forms (introduced in \cite{Dung},\cite{Wang-Li}) by
\begin{equation*}
  \mathcal{H}^k_{f}(M):=\left\{\omega\in\Omega^k(M)~|~d\omega=0, \ \delta_{f}\omega=0\right\}.
\end{equation*}
Observe that \(\mathcal{F}_{f}[\omega] = 0\) if and only if $\omega\in{H}^k_{f}(M)$, that is, the minimum is zero and is attained for the weighted harmonic $k$-forms. For a nonzero infimum, we consider the space
\begin{equation*}
  A^{k}_{\mu}(M):=\Big\{\omega\in
  \mathcal{W}^{1,p}_{\mu}(\Omega^k(M))\Big|\int_{M}|\omega|^p d\mu=1,\int_{M}|\omega|^{p-2}\langle\omega,\varphi\rangle d\mu=0,\forall~\varphi\in\mathcal{H}^k_{f}(M)\Big\}
\end{equation*}
where $\mathcal{W}^{1,p}_{\mu}(\Omega^k(M))$ is the weighted $(1,p)$-Sobolev space of differential $k$-forms. Computing the Euler-Langrange equation of
\eqref{eqn:1.6} leads us to the definition of the following operator:
\begin{definition}
For any integer $p\geq 2$, the weighted $p$-laplace operator is defined as follows
\begin{equation}\label{eqn:1.7}
  \Delta_{p,f}\omega=d(|\delta_{f}\omega|^{p-2}\delta_{f}\omega)
  +\delta_{f}(|d\omega|^{p-2}d\omega), \ \ \ \forall \ \omega\in\Omega^{k}(M).
\end{equation}
\end{definition}

When $f$ is a constant, the weighted $p$-Laplace operator is the $p$-Laplacian $\Delta_p$.
The spectrum of the weighted $p$-Laplacian acting on smooth functions has been studied on compact metric measure spaces
with or without boundaries (see \cite{Du-Mao-Wang-Xia,Sun-Han-Zeng,Wang,Wang-Li}).
The corresponding eigenvalue equation is given by
\begin{equation}\label{eqn:1.8}
  \Delta_{p,f}\omega=\lambda|\omega|^{p-2}\omega, \ \ \ \omega\in\Omega^{k}(M).
\end{equation}
For the first nonzero eigenvalue $\lambda_{1,k,f}$ of the weighted $p$-Laplacian, the variational principle tells us that
\begin{equation*}
  \lambda_{1,k,f}=\inf\big\{\mathcal{F}_{f}[\omega]~
  |~\omega\in{A_{\mu}^k(M)}\big\}.
\end{equation*}

In this paper, we give the following lower bound estimate for the first eigenvalue of the weighted $p$-Laplacian acting on differential forms on a metric measure space having a lower bound on ${\rm Ric}_{N,f}^{[k]}$
(see \eqref{def:2.8} for the precise definition):

\begin{theorem}\label{thm:1.3}
Let $(M, g, e^{-f}dv)$ be an $n$-dimensional closed metric measure space with ${\rm Ric}_{N,f}^{[k]}\geq k(n-k)c$ for some constant $c$ and $p\geq 2$. Then the first nonzero eigenvalue $\lambda_{1,k,f}$ of the weighted $p$-Laplacian $\Delta_{p,f}$ satifies
\begin{equation}\label{1.9}	
  \lambda_{1,k,f}\geq\left(\frac{k(n-k)c}
  {2^{1-\frac{2}{p}}(\widetilde{C}+\frac{p-2}{2})}\right)^{\frac{p}{2}},
\end{equation}
where
$$\widetilde{C}={\rm max}\Big\{\frac{k}{k+1},\frac{N-1}{N}\Big\}.$$
\end{theorem}

When $1\leq k\leq\frac{n}{2}$, we immediately obtain the following corollary:

\begin{corollary}
Let $(M, g, e^{-f}dv)$ be an $n$-dimensional closed metric measure space and $p\geq 2$. If ${\rm Ric}_{N,f}^{[k]}\geq k(n-k)c$ for some constant $c$ and $1\leq k\leq \frac{n}{2}$, then the first nonzero eigenvalue $\lambda_{1,k,f}$ of the weighted $p$-Laplacian $\Delta_{p,f}$ satifies
\begin{equation}\label{1.11}	
  \lambda_{1,k,f}\geq\left(\frac{k(n-k)c}
  {2^{1-\frac{2}{p}}\big(\frac{N-1}{N}+\frac{p-2}{2}\big)}\right)^{\frac{p}{2}}.
\end{equation}
\end{corollary}

\begin{remark}\label{rem:1.1}
When $N=n-k+1$, the conclusion \eqref{1.9} becomes the estimate \eqref{1.3} in Theorem \ref{thm:1.2}. Furthermore,
for $1\leq k\leq\frac{n}{2}$ and $p=2$,  the conclusion
\eqref{1.11} recovers the estimate \eqref{1.1} in Theorem \ref{thm:1.1}.
\end{remark}

Moreover, for $p=2$, we have the following lower bound estimate
(see also Proposition 2.9 in \cite{Branding-Habib}).

\begin{corollary}\label{cor:1.2}
Let $(M, g, e^{-f}dv)$ be an $n$-dimensional closed metric measure space with ${\rm Ric}_{N,f}^{[k]}\geq k(n-k)c$ for some constant $c$.
Then the first positive eigenvalue $\lambda'_{1,k,f}$ of the weighted Hodge Laplacian $\Delta_{f}$ restricted to exact $k$-forms satifies
\begin{equation}\label{1.12}	
  \lambda'_{1,k,f}\geq\frac{N}{N-1}k(n-k)c.
\end{equation}
\end{corollary}

In what follows, we consider the situation where $M$ is a submanifold. In recent years, significant research efforts have been devoted to eigenvalue problems for the Hodge Laplacian on submanifolds. Among these developments, Raulot-Savo \cite{Raulot-Savo} and Savo \cite{Savo} investigated eigenvalues of the Hodge Laplacian on hypersurfaces immersed into a Riemannian manifold. Cui-Sun \cite{Cui-Sun} studied eigenvalues of the Hodge Laplacian on submanifolds of arbitrary codimension in a Riemannian manifold, and obtained the following optimal lower bound for the first eigenvalue of the Hodge Laplacian acting on differential forms:

\begin{theorem}\label{thm:1.4}{\rm (\cite{Cui-Sun})}
Suppose $F:M\rightarrow\overline{M}$ is an isometric immersion from a closed $n$-dimensional Riemannian manifold into an $(n+m)$-dimensional Riemannian manifold $\overline{M}$ with the pull back Weitzenb\"{o}ck curvature operator $F^* {\overline{\rm Ric}}^{[k]} \geq k(n-k)c$  for some constant $c$ and $1\leq k\leq\frac{n}{2}$. Then the first nonzero eigenvalue $\lambda_{1,k}$ of the Hodge Laplacian $\Delta$ acting on $k$-forms of $M$ satisfies
\begin{equation}\label{eqn:1.10}
  \lambda_{1,k}\geq k(n-k+1)(c+\gamma_k),
\end{equation}
where
\begin{equation*}
  \gamma_{k}=\underset{M}{\rm min}\Big\{|H|^2-\frac{1}{n}\big|\mathring{B}\big|^2-\frac{(n-2k)|H|}
  {\sqrt{nk(n-k)}}\big|\mathring{B}\big|\Big\},
\end{equation*}
 $H$ is the mean curvature vector and $\mathring B$ is the traceless part of the second fundamental form $B$. Moreover, if $M$ is totally umbilical, then
 $$\lambda_{1,k}\geq k(n-k+1)\big(c+\underset{M}{\mathrm {min}}|H|^2\big).$$
\end{theorem}

Inspired by the above result, we obtain the following eigenvalue estimate for the weighted $p$-Laplacian on closed submanifolds of a metric measure space.

\begin{theorem}\label{thm:1.5}
Suppose $F:M\rightarrow\overline{M}$ is an isometric immersion from a closed $n$-dimensional Riemannian manifold into an $(n+m)$-dimensional
Riemannian manifold. Assume that $f$ is a smooth function on $\overline{M}$ and $F^*\overline{\rm Ric}_{N,f}^{[k]}\geq k(n-k)c$ for some constant $c$ and $1\leq k\leq\frac{n}{2}$. If $p\geq 2$, then the first nonzero eigenvalue $\lambda_{1,k,f}$ of the weighted $p$-Laplacian $\Delta_{p,f}$
acting on $k$-forms of $M$ satifies
\begin{equation}\label{1.14}
  \lambda_{1,k,f}\geq\bigg(\frac{k(n-k)(c+\widetilde{\gamma}_{k})}
  {2^{1-\frac{2}{p}}(\widetilde{C}+\frac{p-2}{2})}
  \bigg)^{\frac{p}{2}},
\end{equation}
where
\begin{equation*}
\begin{aligned}
\widetilde{C}&={\rm max}\Big\{\frac{k}{k+1},\frac{N-1}{N}\Big\},\qquad
\xi=\underset{M}{\rm max}\big|(\overline{\nabla}f)^{\bot}\big|,\\
\widetilde{\gamma}_{k}&=\min_{M}\Big\{|H|^2-\frac{1}{n}\big|\mathring{B}\big|^2
-\frac{(n-2k)|H|+\xi}{\sqrt{nk(n-k)}}\big|\mathring{B}\big|
-\frac{\xi}{n-k}|H|\Big\},
\end{aligned}
\end{equation*}
$H$ is the mean curvature vector and $\mathring{B}$ is the traceless part of the second fundamental form $B$. Moreover, if $M$ is totally umbilical, then
\begin{equation*}
  \lambda_{1,k,f}\geq\left(\frac{k(n-k)}
  {2^{1-\frac{2}{p}}(\widetilde{C}+\frac{p-2}{2})}
  \underset{M}\min\Big(c+|H|^2-\frac{\xi}{n-k}|H|\Big)\right)^{\frac{p}{2}}.
\end{equation*}
\end{theorem}

When $N=n-k+1$, $f$ is a constant, we immediately obtain the following corollary:

\begin{corollary}\label{Cor:1.3}
Suppose $F:M\rightarrow\overline{M}$ is an isometric immersion from a closed $n$-dimensional Riemannian manifold into an $(n+m)$-dimensional Riemannian manifold $\overline{M}$ with the pull back Weitzenb\"{o}ck curvature operator $F^*{\overline{\rm Ric}}^{[k]} \geq k(n-k)c$  for some constant $c$ and $1\leq k\leq\frac{n}{2}$. If $p\geq 2$, then the first nonzero eigenvalue $\lambda_{1,k}$ of the $p$-Laplacian $\Delta_{p}$
acting on $k$-forms of $M$ satifies
\begin{equation}\label{1.15}
  \lambda_{1,k}\geq\bigg(\frac{k(n-k)(c+\gamma_{k})}
  {2^{1-\frac{2}{p}}\big(\frac{n-k}{n-k+1}+\frac{p-2}{2}\big)}
  \bigg)^{\frac{p}{2}},
\end{equation}
where
\begin{equation*}
\gamma_{k}=\min_{M}\Big\{|H|^2-\frac{1}{n}\big|\mathring{B}\big|^2
-\frac{(n-2k)|H|}{\sqrt{nk(n-k)}}\big|\mathring{B}\big|\Big\},
\end{equation*}
 $H$ is the mean curvature vector and $\mathring B$ is the traceless part of the second fundamental form $B$. Moreover, if $M$ is totally umbilical, then
 \begin{equation*}
  \lambda_{1,k}\geq\bigg(\frac{k(n-k)}
  {2^{1-\frac{2}{p}}\big(\frac{n-k}{n-k+1}+\frac{p-2}{2}\big)}
  \left(c+\underset{M}\min|H|^2\right)\bigg)^{\frac{p}{2}}.
\end{equation*}
\end{corollary}

\begin{remark}\label{rem:1.3}
When $p=2$, the above conclusion \eqref{1.15} covers the estimate \eqref{eqn:1.10} in Theorem \ref{thm:1.4}. Furthermore, if
 codimension $m=1$, Savo \cite{Savo} obtained Corollary \ref{Cor:1.3}.
\end{remark}

\section{Preliminaries}\label{sect:2}

In this section, we explore properties related to the weighted Hodge Laplacian acting on differential forms which are needed in the proofs of our results.

Let $(M,g)$ be an $n$-dimensional oriented Riemannian manifold, and $dv$ is the Riemannian volume element associated to the Riemannian metric $g$.
Let $\Omega^k(M)$ be the space of smooth differential forms of degree $k$ on $M$.
Denote by $\{\theta^1,\cdots,\theta^{n}\}$ the dual coframe of an orthonormal frame $\{e_1,\cdots,e_{n}\}$ on $M$.
For any $k$-form $\omega$, the Weitzenb\"{o}ck curvature operator (or the Bochner curvature operator)
 ${\rm Ric}^{[k]}:\Omega^k(M)\rightarrow\Omega^k(M)$ is given by
\begin{equation}\label{w2.1}
  {\rm Ric}^{[k]}(\omega)=\sum_{i,j=1}^{n}\theta^j\wedge\iota_{e_i}R(e_i, e_j)\omega,
\end{equation}
where $R$ denotes the Riemann curvature operator. When $k=1$,  ${\rm Ric}^{[1]}$ is just the Ricci tensor.
Then we have the following two formulas:
\begin{equation}\label{2.1}
\Delta\omega
=\nabla^*\nabla\omega+{\rm Ric}^{[k]}(\omega),
\end{equation}
\begin{equation}\label{2.2}
\frac{1}{2}\Delta|\omega|^2
=\langle\Delta\omega,\omega\rangle-|\nabla\omega|^2-\langle {\rm Ric}^{[k]}(\omega),\omega\rangle,
\end{equation}
where $\nabla$ is the Levi-Civita connection and
$\nabla^*\nabla$ is the connection Laplacian on $M$.
Equalities \eqref{2.1} and \eqref{2.2} are usually called Weitzenb\"{o}ck formula and Bochner formula.
By the work \cite{Gallot-Meyer1} of Gallot and Meyer, if the eigenvalues of the Riemann curvature operator $R$ are bounded from below by $c\in\mathbb{R}$, then
\begin{equation}\label{eqn:2.3}
 \langle{\rm Ric}^{[k]}(\omega),\omega\rangle \geq k(n-k)c|\omega|^2.
\end{equation}


Let $f$ be a smooth real-valued function on $M$. When the measure is changed from being $dv$ to $d\mu=e^{-f}dv$,
the triple $(M, g,d\mu)$ is called a metric measure space (a weighted manifold or a Bakry-\'{E}mery manifold).
Define
\begin{equation}\label{def:2.5}
\delta_{f}:=\delta+\iota_{\nabla f},
\qquad
\nabla^{*}_f:=\nabla^{*}+\iota_{\nabla f}.
\end{equation}
By the Stokes formula, we immediately obtain the following formula:

\begin{proposition}\label{prop:2.1}
Let $(M,g,e^{-f}dv)$ be a closed metric measure space. For any two differential forms $\omega\in\Omega^k(M)$ and $\phi\in\Omega^{k+1}(M)$, there holds
\begin{equation}\label{eqn:2.1}
\int_M\langle d\omega,\phi\rangle  d\mu=\int_M\langle\omega,\delta_f\phi \rangle d\mu,
\end{equation}
where $d\mu=e^{-f}dv.$
\end{proposition}

By a straightforward computation, we derive the following weighted Weitzenb\"{o}ck formula \cite{Petersen-Wink} and weighted Bochner formula
\cite{Dung-Sung}:
\begin{equation}\label{eqn:2.2}
  \Delta_f\omega
 =\nabla_f^*\nabla\omega+{\rm Ric}_f^{[k]}(\omega),
\end{equation}
\begin{equation}\label{eqn:2.5}
  \frac{1}{2}\Delta_f |\omega|^2
  =\langle\Delta_f\omega, \omega\rangle-|\nabla\omega|^2
  -\langle{\rm Ric}_f^{[k]}(\omega), \omega\rangle,
\end{equation}
where ${\rm Ric}_f^{[k]}$ is called the weighted Weitzenb\"{o}ck curvature
 operator defined by
\begin{equation*}
  {\rm Ric}_f^{[k]}(\omega)={\rm Ric}^{[k]}(\omega)+\nabla^2{f}(\omega).
\end{equation*}
Here $\nabla^2{f}$ is the Hessian acting on forms \cite{Escobar-Freire} given by
 \begin{equation}\label{hess}
\nabla^2{f}(\omega)
:=\sum_{i,j=1}^{n}\nabla^2{f}(e_i,e_j)\theta^j
\wedge\iota_{e_i}\omega.
\end{equation}
For $k=1$, the weighted Weitzenb\"{o}ck curvature operator
 is just ${\rm Ric}^{[1]}_f={\rm Ric}+\nabla^2f$
which is the Bakry-\'{E}mery Ricci tensor (or the $\infty$-Bakry-\'{E}mery Ricci tensor) which is introduced by Bakry-\'{E}mery \cite{Bakry-Emery} in the study of diffusion processes (see also \cite{Bakry-Gentil-Ledoux} for a comprehensive introduction), and then it has been extensively investigated in the theory of the Ricci flow.
Also, for any constant $N\geq n-k+1$, the $N$-dimensional Bakry-\'{E}mery Ricci operator
${\rm Ric}_{N,f}^{[k]}:\Omega^k(M)\rightarrow\Omega^k(M)$ (introduced in
\cite{Branding-Habib}) is defined by
\begin{equation}\label{def:2.8}
{\rm Ric}_{N,f}^{[k]}(\omega):={\rm Ric}_{f}^{[k]}(\omega)-\frac{1}{N-(n-k+1)}df\wedge
\iota_{\nabla f}\omega,
\end{equation}
and when $N=n-k+1$,  ${\rm Ric}_{n-k+1,f}^{[k]}$ is defined if and only if $f$ is constant. For $k=1$, ${\rm Ric}_{N,f}^{[1]}$
 corresponds to the classical $N$-dimensional Bakry-\'{E}mery Ricci tensor.

To prove Theorem \ref{thm:1.3} and Theorem \ref{thm:1.5}, we need the following estimate:
\begin{proposition}\label{prop:2.4}
Let $\omega\in\Omega^k(M)$.
Then, for any constant $N>n-k+1$, we have
\begin{equation}\label{eqn:2.6}
  |\nabla\omega|^2
  \geq\frac{1}{k+1}|d\omega|^2+\frac{1}{N}|\delta_f\omega|^2
  -\frac{1}{N-(n-k+1)}|\iota_{\nabla f}\omega|^2.
\end{equation}
\end{proposition}

\begin{proof}
For any smooth tangent vector field $X$ on $M$, we have the following orthogonal decomposition \cite{Gallot-Meyer}:
\begin{equation}\label{2.12}
  \nabla_{X}\omega
  =\frac{1}{k+1}\iota_{X}d\omega-\frac{1}{n-k+1}X^{\flat}\wedge\delta\omega
+T\omega(X),
\end{equation}
where $X^{\flat}$ is the dual $1$-form defined by $X^{\flat}(e_i)=g(X,e_i)$, and $T$ is the twistor operator on $M$.
Then
\begin{equation*}
\begin{aligned}
  |\nabla\omega|^2
&=\frac{1}{k+1}|d\omega|^2+\frac{1}{n-k+1}|\delta\omega|^2
+|T\omega|^2\\
&\geq\frac{1}{k+1}|d\omega|^2+\frac{1}{n-k+1}|\delta\omega|^2\\
&=\frac{1}{k+1}|d\omega|^2+\frac{1}{n-k+1}|\delta_{f}\omega-\iota_{\nabla f}\omega|^2\\
&\geq\frac{1}{k+1}|d\omega|^2+\frac{1}{n-k+1}
\big(|\delta_{f}\omega|-|\iota_{\nabla f}\omega|\big)^2\\
&\geq\frac{1}{k+1}|d\omega|^2+\frac{1}{N}|\delta_f\omega|^2
  -\frac{1}{N-(n-k+1)}|\iota_{\nabla f}\omega|^2,
 \end{aligned}
\end{equation*}
where we used the inequality
$$\frac{(a-b)^2}{n-k+1}\geq\frac{1}{N}a^2-\frac{1}{N-(n-k+1)}b^2$$
for all $N>n-k+1$.
\end{proof}

\section{Variational characterization of the eigenvalue}\label{sect:3}

In this section, we will compute the Euler-Lagrange equation of \eqref{eqn:1.6} and show that
the infimum can be characterized as the eigenvalue problem \eqref{eqn:1.8}.

Let $(M, g, d\mu=e^{-f}dv)$ be an $n$-dimensional closed metric measure space.
For $0\leq k\leq n$ and $p\geq 2$, we denote by $L^p_{\mu}\big(\Omega^k(M)\big)$ the space of measurable $k$-forms on $M$
satisfying
$\int_M |\omega|^pd\mu<\infty.$
Then, we use $\mathcal{W}^{1,p}_{\mu}\big(\Omega^k(M)\big)$ to denote
the weighted Sobolev space \cite{Schwarz} of differential $k$-forms in $L^p_{\mu}\big(\Omega^k(M)\big)$
with generalized gradient \cite{Scott} in $L^p_{\mu}\big(\Omega^k(M)\big)$.

\begin{definition}
We call that $\lambda$ is an eigenvalue of the weighted $p$-Laplacian
$\Delta_{p,f}$, if there exists a $k$-form $\omega\in {W}^{1,p}_{\mu}(\Omega^{k}({M}))$ such that
\begin{equation*}
  \Delta_{p,f}\omega=\lambda|\omega|^{p-2}\omega
\end{equation*}
 in distribution sense.
 Namely, for any $\phi\in\Omega^{k}(M)$,
\begin{equation*}
  \int_{M}|d\omega|^{p-2}\langle d\omega,d\phi\rangle d\mu
+\int_{M}|\delta_{f}\omega|^{p-2}\langle\delta_{f}\omega,\delta_{f}\phi\rangle
 d\mu=\lambda\int_{M}|\omega|^{p-2}\langle\omega,\phi\rangle d\mu.
\end{equation*}
\end{definition}

First, we show that the first nonzero eigenvalue $\lambda_{1,k,f}$ can be characterized as the infinum of the weighted $L^p$-Dirchlet integral \eqref{eqn:1.6} over the space $A^{k}_{\mu}(M)$.
\begin{proposition}\label{prop:3.1}
For a closed metric measure space $(M, g, e^{-f}dv)$ and $p\geq 2$,
the first nonzero eigenvalue $\lambda_{1,k,f}$ of the weighted $p$-Laplacian $\Delta_{p,f}$ satifies
\begin{equation*}
 \lambda_{1,k,f}
=\inf\Big\{\int_{M}\big(|d\omega|^{p}+|\delta_{f}\omega|^p\big)d\mu
~ \Big |~\omega\in A_{\mu}^k(M)\Big\}.
\end{equation*}
\end{proposition}

\begin{proof}
 Assume that $\omega\in A_{\mu}^{k}(M)$ minimizes
 the functional \eqref{eqn:1.6} over the space $A^{k}_{\mu}(M)$.
 Let $\phi(t)\in A_{\mu}^{k}(M)$ be a variation of $\omega$ such that $\phi(0)=\omega.$
 By the method of Lagrange multipliers, we set
 \begin{equation*}
 \mathcal{L}(t,\lambda,\tau)=\mathcal{F}_{f}[\phi(t)]
 +\lambda\Big(1-\int_{M}|\phi(t)|^p d\mu\Big)+\tau\Big(\int_{M}|\phi(t)|^{p-2}\langle\phi(t),\varphi \rangle d\mu\Big),
 \end{equation*}
where $\lambda,\tau$ are two Lagrange multipliers,  $\varphi\in\mathcal{H}^k_{f}(M)$ is any weighted harmonic $k$-form.
Then, we have
\begin{equation*}
 \begin{aligned}
 0&=\frac{\partial\mathcal{L}}{\partial t}\bigg|_{t=0}\\ &=p\int_{M}\Big(|d\omega|^{p-2}\langle{d}\omega,d\phi'(0)\rangle
 +|\delta_{f}\omega|^{p-2}\langle\delta_{f}\omega,
 \delta_{f}\phi'(0)\rangle
 -\lambda|\omega|^{p-2}\langle\omega,\phi'(0)\rangle\Big)d\mu\\
 & \ \ \ +\tau\int_M \Big((p-2)|\omega|^{p-4}\langle\omega,\phi'(0)\rangle\langle\omega,
\varphi\rangle+|\omega|^{p-2}\langle\phi'(0),\varphi\rangle\Big)d\mu.
 \end{aligned}
\end{equation*}

Setting $\phi'(0)=\varphi$, we get
\begin{equation*}
 \begin{aligned}
0=\tau\int_{M}\Big((p-2)|\omega|^{p-4}\langle\omega,\varphi\rangle^2
+|\omega|^{p-2}|\varphi|^2\Big)d\mu,
 \end{aligned}
\end{equation*}
so that $\tau=0$. Therefore, for any variation $\phi(t)$ of $\omega$,
we obtain
\begin{equation*}
\begin{aligned}
\int_{M}\Big(|d\omega|^{p-2}\langle{d}\omega,d\phi'(0)\rangle
 +|\delta_{f}\omega|^{p-2}\langle\delta_{f}\omega,
 \delta_{f}\phi'(0)\rangle\Big)d\mu
 =\lambda\int_{M}|\omega|^{p-2}\langle\omega,\phi'(0)\rangle d\mu,
\end{aligned}
\end{equation*}
that is
\begin{equation*}
 \begin{aligned}
  &\Delta_{p,f}\omega=\lambda|\omega|^{p-2}\omega.
 \end{aligned}
\end{equation*}
We finished the proof of the proposition.
\end{proof}

\section{Eigenvalue estimates on closed manifolds}

In this section, we present the proof of Theorem \ref{thm:1.3}. Let $\omega$ be an eigenform associated with $\lambda=\lambda_{1,k,f}$. By the weighted Bochner formula \eqref{eqn:2.5}, we derive
\begin{equation}\label{eqn:4.1}
 \begin{aligned}
  & \ \ \ \ \ \ \int_{M}\langle\Delta_{p,f}\omega,\Delta_{f}\omega\rangle{d\mu}\\
  &=\lambda\int_{M}|\omega|^{p-2}\langle\omega,
  \Delta_{f}\omega\rangle{d\mu}\\		
  &=\lambda\int_{M}|\omega|^{p-2}
  \Big(\frac{1}{2}\Delta_{f}|\omega|^{2}+|\nabla\omega|^2
  +\langle{\rm Ric}^{[k]}_{f}(\omega),\omega\rangle\Big){d\mu}\\				&=-\frac{\lambda}{2}\int_{M}{\rm div}_{f}(|\omega|^{p-2}\nabla|\omega|^{2})d\mu
  +\frac{\lambda}{2}\int_{M}\langle\nabla|\omega|^{p-2},
  \nabla|\omega|^{2}\rangle{d\mu}\\
  & \ \ \ \ +\lambda\int_{M}|\omega|^{p-2}\Big(|\nabla\omega|^2
  +\langle{\rm Ric}^{[k]}_{f}(\omega),\omega\rangle\Big)d\mu\\		
  &=\lambda\int_{M}|\omega|^{p-2}\Big((p-2)
  \big|\nabla|\omega|\big|^{2}+|\nabla\omega|^2
  +\langle{\rm Ric}^{[k]}_{f}(\omega),\omega\rangle\Big)d\mu,
\end{aligned}
\end{equation}
where ${\rm div}_{f}(\cdot):=e^f{\rm div}(e^{-f}\cdot)$.

On the other hand, using Proposition \ref{prop:2.1}, we have
\begin{equation}\label{eqn:4.2}
 \begin{aligned}
  & \ \ \ \ \ \ \int_{M}\langle\Delta_{p,f}\omega,\Delta_{f}\omega\rangle{d\mu}\\			&=\lambda\int_{M}\langle|\omega|^{p-2}\omega,
  d\delta_{f}\omega\rangle{d\mu}
  +\lambda\int_{M}\langle|\omega|^{p-2}\omega,
  \delta_{f}d\omega\rangle{d\mu}\\	&=\lambda\int_{M}\langle\delta_{f}(|\omega|^{p-2}\omega),
  \delta_{f}\omega\rangle{d\mu}
  +\lambda\int_{M}\langle{d}(|\omega|^{p-2}\omega),
  d\omega\rangle{d\mu}\\			
  &=\lambda\int_{M}|\omega|^{p-2}|\delta_{f}
  \omega|^{2}d\mu+\lambda\int_{M}|\omega|^{p-2}|d\omega|^{2}d\mu\\
  & \ \ \ \ -\lambda\int_{M}\langle\iota_{\nabla|\omega|^{p-2}}\omega,
  \delta_{f}\omega\rangle{d\mu}
  +\lambda\int_{M}\langle{d}(|\omega|^{p-2})\wedge\omega,
  d\omega\rangle{d}\mu.
 \end{aligned}
\end{equation}
Combining \eqref{eqn:4.1} and \eqref{eqn:4.2}, we obtain
\begin{equation}\label{eqn:4.3}
 \begin{aligned}			
 & \ \ \ \
 \int_{M}|\omega|^{p-2}\Big((p-2)\big|\nabla|\omega|\big |^{2}+|\nabla\omega|^2+\langle{\rm Ric}^{[k]}_{f}(\omega),
 \omega\rangle\Big)d\mu\\
&=\int_{M}|\omega|^{p-2}\big(|d\omega|^{2}
+|\delta_{f}\omega|^{2}\big)d\mu
+\int_{M}\langle{d}(|\omega|^{p-2})\wedge\omega,d\omega\rangle{d}\mu\\
& \ \ \ \ -\int_{M}\langle\iota_{\nabla|\omega|^{p-2}}\omega,
\delta_{f}\omega\rangle{d\mu}.
 \end{aligned}	
\end{equation}

Next, we estimate the last two terms in \eqref{eqn:4.3} as follows:
\begin{equation}\label{eqn:4.4}
 \begin{aligned}	
    \int_{M}\langle{d}(|\omega|^{p-2})\wedge\omega,d\omega\rangle{d\mu}
  &=\int_{M}\langle\omega,
  \iota_{\nabla|\omega|^{p-2}}d\omega\rangle{d\mu}\\
  &\leq\int_{M}\big|\nabla|\omega|^{p-2}\big|\cdot
  |d\omega|\cdot|\omega|d\mu\\
  &=(p-2)\int_{M}|\omega|^{p-2}\cdot\big|\nabla|\omega|\big|
  \cdot|d\omega|d\mu\\
  &\leq\frac{(p-2)}{2}\int_{M}|\omega|^{p-2}\Big(\big |\nabla|\omega|\big|^{2}+|d\omega|^{2}\Big)d\mu,
 \end{aligned}
\end{equation}
and
\begin{equation}\label{eqn:4.5}
 \begin{aligned}
  -\int_{M}\langle\iota_{\nabla|\omega|^{p-2}}\omega,
  \delta_{f}\omega\rangle d\mu
  &\leq\int_{M}|\omega|\cdot\big|\nabla|\omega|^{p-2}\big|
  \cdot|\delta_{f}\omega|d\mu\\
  &=(p-2)\int_{M}|\omega|^{p-2}\cdot\big|\nabla|\omega|\big|
  \cdot|\delta_{f}\omega|d\mu\\
  &\leq\frac{(p-2)}{2}\int_{M}|\omega|^{p-2}
  \Big(\big|\nabla|\omega|\big|^{2}
  +|\delta_{f}\omega|^{2}\Big)d\mu.
 \end{aligned}
\end{equation}

Applying the above estimates to \eqref{eqn:4.3} and using
Proposition \ref{prop:2.4}, we get
\begin{equation}\label{ineq:4.6}
 \begin{aligned}
   & \ \ \ \
  \Big(\frac{p-2}{2}+1\Big)
  \int_{M}|\omega|^{p-2}\Big(|d\omega|^2+|\delta_{f}\omega|^2\Big)d\mu\\
  &\geq\int_M|\omega|^{p-2}|\nabla\omega|^2{d\mu}
  +\int_M|\omega|^{p-2}\langle{\rm Ric}^{[k]}_{f}(\omega),\omega\rangle{d\mu}\\		&\geq\frac{1}{k+1}\int_M|\omega|^{p-2}|d\omega|^2{d\mu}
  +\frac{1}{N}\int_M|\omega|^{p-2}|\delta_{f}\omega|^2{d\mu}\\
  & \ \ \ \ -\frac{1}{N-(n-k+1)}\int_M|\omega|^{p-2}
  |\iota_{\nabla f}\omega|^2{d\mu}
  +\int_M|\omega|^{p-2}\langle{\rm Ric}_{f}^{[k]}(\omega),\omega\rangle{d\mu}\\
  &=\frac{1}{k+1}\int_M|\omega|^{p-2}|d\omega|^2{d\mu}
  +\frac{1}{N}\int_M|\omega|^{p-2}|\delta_{f}\omega|^2{d\mu}\\
  & \ \ \ \ +\int_M|\omega|^{p-2}\langle{\rm Ric}_{N,f}^{[k]}(\omega),\omega\rangle{d\mu}.
  \end{aligned}	
\end{equation}
Notice that
\begin{equation}\label{eqn:4.6}		
 \begin{aligned}
  \int_M | \omega|^{p-2} |d\omega|^2{d\mu}
  \leq \Big(\int_M|\omega|^p{d\mu}\Big)^{1-\frac{2}{p}}
  \Big(\int_M|d\omega|^p{d\mu}\Big)^\frac{2}{p}
 \end{aligned}
\end{equation}
and
\begin{equation}\label{eqn:4.7}
 \begin{aligned}
  \int_M |\omega|^{p-2} |\delta_f{\omega}|^2{d\mu}\leq \Big(\int_M|\omega|^p{d\mu}\Big)^{1-\frac{2}{p}}
  \Big(\int_M|\delta_f{\omega}|^p{d\mu}\Big)^{\frac{2}{p}}.
 \end{aligned}
\end{equation}
Let
$$\widetilde{C}={\rm max}\Big\{\frac{k}{k+1},\frac{N-1}{N}\Big\}.$$
By \eqref{eqn:4.6} and \eqref{eqn:4.7}, we have
\begin{equation}	
 \begin{aligned}
  & \ \ \ \
  	\int_M|\omega|^{p-2}\langle{\rm Ric}_{N,f}^{[k]}(\omega),\omega\rangle{d\mu}\\
  &\leq\Big(\widetilde{C}+\frac{p-2}{2}\Big)\int_M|\omega|^{p-2}
  \big(|d\omega|^{2}+|\delta_{f}\omega|^{2}\big)d\mu\\	 &\leq\Big(\widetilde{C}+\frac{p-2}{2}\Big)
  \Big(\int_M|\omega|^{p}d\mu\Big)^{1-\frac{2}{p}}
  \left[\Big(\int_{M}|d\omega|^{p}d\mu\Big)^{\frac{2}{p}}
  +\Big(\int_{M}|\delta_{f}\omega|^{p}d\mu\Big)^{\frac{2}{p}}\right]\\			&\leq2^{1-\frac{2}{p}}\Big(\widetilde{C}+\frac{p-2}{2}\Big)
  \Big(\int_M|\omega|^{p}d\mu\Big)^{1-\frac{2}{p}}
  \left(\int_{M}|d\omega|^{p}d\mu
  +\int_{M}|\delta_{f}\omega|^{p}d\mu\right)^{\frac{2}{p}}.	
 \end{aligned}
\end{equation}
Using the fact that $\int_{M}|d\omega|^{p}d\mu+\int_{M}|\delta_{f}\omega|^{p}d\mu
=\lambda\int_{M}|\omega|^{p}d\mu$, we get
\begin{equation}\label{4.10}
\begin{aligned}		
  2^{1-\frac{2}{p}}(\widetilde{C}+\frac{p-2}{2})
  \lambda^{\frac{2}{p}}\int_M|\omega|^{p}d\mu
  &\geq\int_M|\omega|^{p-2}\langle{\rm Ric}_{N,f}^{[k]}(\omega),\omega\rangle{d\mu}.
\end{aligned}
\end{equation}
This completes the proof of Theorem \ref{thm:1.3}. \qed

\begin{proof}[Proof of Corollary \ref{cor:1.2}]
Taking $p=2$ in \eqref{ineq:4.6}, we immediately obtain the estimate
\eqref{1.12}.
\end{proof}

\section{Eigenvalue estimates on closed submanifolds}

In this section, we first recall some fundamental facts on
the geometry of submanifolds.

Let $F:M\to\overline{M}$ be an isometric immersion from an $n$-dimensional closed Riemannian manifold $(M,g)$ into an $(n+m)$-dimensional Riemannian manifold $(\overline{M},\overline{g})$.
We choose an orthonormal frame $\{e_1,\cdots,e_{n+m}\}$ on $\overline{M}$ and the dual coframe $\{\theta^1,\cdots,\theta^{n+m}\}$ such that, restricted to $M$, $e_1,\cdots,e_n$ are tangent to $M$, and $e_{n+1},\cdots,e_{n+m}$
are normal to $M$.
We will agree on the following index convention:
$$ 1\leq i,j,k,\cdots\leq n;\quad
 n+1\leq\alpha,\beta,\gamma,\cdots\leq n+m.$$

Let $h_{ij}^{\alpha}$ be the components of the second fundamental form $B$ of the immersion $F:M\to\overline{M}$. Then
the second fundamental form $B$
and the mean curvature vector $H$
are given by
\begin{equation}
  B=h_{ij}^{\alpha}\theta^i\otimes\theta^j\otimes{e_{\alpha}},\quad
  H=H^{\alpha}e_{\alpha}=\frac{1}{n}\sum_{i}h_{ii}^{\alpha}e_{\alpha}.
\end{equation}
We write $\mathring{B}=B-H\otimes g$ which is the traceless part of $B$.
Let $A$ be the shape operator defined by
\begin{equation}
  \langle{B}(X,Y),\nu\rangle
 =\langle A^{\nu}(X), Y\rangle, \qquad
 \forall \ X, Y \in TM, \ \nu\in T^{\bot}M.
\end{equation}
The shape operator admits a canonical extension $S^{[k]}$ acting on $k$-forms of $M$. Explicitly, for any $\omega\in\Omega^{k}(M)$,
\begin{equation}\label{def5.4}
S^{[k]}(\omega)
:=\theta^i\wedge{\iota_{A^{e_\alpha}(e_i)}}
\omega\otimes{e_\alpha}.
\end{equation}
For simplification, we introduce the following notations:
\begin{equation*}
\begin{aligned}
   & A^{\alpha}:=A^{e_{\alpha}},\qquad
\mathring{A}^{\alpha}:=A^{\alpha}-H^{\alpha}g,\\
   &\mathring{S}^{[k]}(\omega):=\theta^i\wedge{\iota_{\mathring {A}^{\alpha}(e_i)}}\omega\otimes{e_\alpha}.
\end{aligned}
\end{equation*}

We denote by $\overline{R}$, $\overline{{\rm Ric}}^{[k]}$ and $\overline{{\rm Ric}}_{N,f}^{[k]}$ the Riemann curvature tensor,
the Weitzenb\"{o}ck curvature operator and the $N$-dimensional Bakry-\'{E}mery Ricci operator
on $\overline{M}$, respectively, while by $R$, ${\rm Ric}^{[k]}$ and ${\rm Ric}_{N,f}^{[k]}$
the Riemann curvature tensor,
the Weitzenb\"{o}ck curvature operator and the $N$-dimensional Bakry-\'{E}mery Ricci operator
on $M$, respectively.
Then the Gauss equations are
\begin{equation}\label{Gauss}
R_{ijkl}=\overline{R}_{ijkl}+\sum_{\alpha}(h_{ik}^{\alpha}
h_{jl}^{\alpha}-h_{il}^{\alpha}
h_{jk}^{\alpha}).
\end{equation}
For any $\omega\in\Omega^{k}(M)$, the Weitzenb\"{o}ck curvature operator ${\rm Ric}^{[k]}$ and the pull back Weitzenb\"{o}ck curvature operator $F^*\overline{{\rm Ric}}^{[k]}$
are given by
\begin{equation*}
\begin{aligned}
  {\rm Ric}^{[k]}(\omega)
&=\sum_{i,j}\theta^j\wedge\iota_{e_i}R(e_i, e_j)\omega
=\sum_{i,j,l,r}R_{ijrl}\theta^j\wedge
\iota_{e_i}(\theta^r\wedge\iota_{e_l}\omega),\\
  F^*\overline{{\rm Ric}}^{[k]}(\omega)
&=\sum_{i,j}\theta^j\wedge\iota_{e_i}\overline{R}(e_i, e_j)\omega
=\sum_{i,j,l,r}\overline{R}_{ijrl}\theta^j\wedge
\iota_{e_i}(\theta^r\wedge\iota_{e_l}\omega).
\end{aligned}
\end{equation*}

\vskip 2mm
\noindent{\bf Proof of Theorem \ref{thm:1.5}} \
For a closed submanifold $M$ of $\overline{M}$, the proof of Theorem \ref{thm:1.3} yields the following estimate:
\begin{equation}\label{eqn:5.1}		
  2^{1-\frac{2}{p}}\big(\widetilde{C}+\frac{p-2}{2}\big)
  \lambda^{\frac{2}{p}}\int_M|\omega|^{p}d\mu
\geq\int_M|\omega|^{p-2}\langle{\rm Ric}_{N,f}^{[k]}(\omega),\omega\rangle{d\mu},
\end{equation}
where $\omega$ is an eigenform associated with
the first nonzero eigenvalue $\lambda=\lambda_{1,k,f}$ of the weighted $p$-Laplacian $\Delta_{p,f}$ acting on $k$-forms of $M$.

Next we consider the curvature term $\langle{\rm Ric}_{N,f}^{[k]}(\omega),\omega\rangle$.
We denote by $\overline{\nabla}$ and $\overline{\nabla}^2$ the gradient and the Hessian on $\overline{M}$, respectively, while by $\nabla$ and $\nabla^2$ the gradient and the Hessian on $M$, respectively.
By the definition \eqref{def:2.8}, we have
\begin{equation}\label{5.8}	
 \langle{\rm Ric}_{N,f}^{[k]}(\omega),\omega\rangle
=\langle{\rm Ric}^{[k]}(\omega),\omega\rangle
   +\langle\nabla^{2}f(\omega),\omega\rangle
   -\frac{1}{N-(n-k+1)}|\iota_{\nabla f}\omega|^2.
\end{equation}
It follows from \eqref{Gauss} that
\begin{equation}\label{5.6}
\langle{\rm Ric}^{[k]}(\omega),\omega\rangle
=\langle F^*\overline{{\rm Ric}}^{[k]}(\omega),\omega\rangle
+\langle S^{[k]}(\omega),nH\omega\rangle
-\big|S^{[k]}(\omega)\big|^2.
\end{equation}
Notice that
\begin{equation*}
\begin{aligned}
\nabla f&=(\overline{\nabla}f)^{\top},\\
\nabla^2f(e_i,e_j)&=\overline{\nabla}^2f(e_i,e_j)
+\langle\overline{\nabla}f,B(e_i,e_j)\rangle,
\end{aligned}
\end{equation*}
where $\top$ denotes the tangent projection to $M$.
Hence, by \eqref{hess}, we obtain
\begin{equation}\label{5.9}
 \begin{aligned}
& \ \ \ \ \langle\nabla^2{f}(\omega),\omega\rangle\\
&=\sum_{i,j}\nabla^2{f}(e_i,e_j)\langle\theta^j
  \wedge\iota_{e_i}\omega,\omega\rangle\\
&=\sum_{i,j}\overline{\nabla}^2{f}(e_i,e_j)\langle\theta^j
  \wedge\iota_{e_i}\omega,\omega\rangle
  +\sum_{i,j}\langle\overline{\nabla}f,B(e_i,e_j)\rangle
  \langle\theta^j\wedge\iota_{e_i}\omega,\omega\rangle\\
&=\langle F^*\overline{\nabla}^2{f}(\omega),\omega\rangle
  +\sum_{i,j}\sum_{\alpha}
  \langle\overline{\nabla}f,e_{\alpha}\rangle
  \langle A^{\alpha}(e_j),e_i\rangle
  \langle\theta^j\wedge\iota_{e_i}\omega,\omega\rangle\\
&=\langle F^*\overline{\nabla}^2{f}(\omega),\omega\rangle
  +\sum_{j}\sum_{\alpha}
  \langle\overline{\nabla}f,e_{\alpha}\rangle
  \langle\theta^j\wedge\iota_{A^{\alpha}(e_j)}\omega,\omega\rangle\\
 &=\langle F^*\overline{\nabla}^2{f}(\omega),\omega\rangle
  +\langle S^{[k]}(\omega),
  (\overline{\nabla}f)^{\bot}\omega\rangle,
\end{aligned}
\end{equation}
where $\bot$ denotes the projection onto the normal bundle $T^{\bot}M$ of $M$.
Applying \eqref{5.6} and \eqref{5.9} to \eqref{5.8}, we get
\begin{equation*}\label{5.11}
 \begin{aligned}
   & \ \ \ \ \langle{\rm Ric}_{N,f}^{[k]}(\omega),\omega\rangle\\
   &=\langle F^*\overline{{\rm Ric}}_{N,f}^{[k]}(\omega),\omega\rangle
+\langle S^{[k]}(\omega),\big(nH+(\overline{\nabla}f)^{\bot}\big)\omega\rangle
-\big|S^{[k]}(\omega)\big|^2\\
   &=\langle F^*\overline{{\rm Ric}}_{N,f}^{[k]}(\omega),\omega\rangle
+\langle\mathring{S}^{[k]}(\omega)+kH\omega,
\big(nH+(\overline{\nabla}f)^{\bot}\big)\omega\rangle
-\big|\mathring{S}^{[k]}(\omega)+kH\omega\big|^2\\
   &=\langle F^*\overline{{\rm Ric}}_{N,f}^{[k]}(\omega),\omega\rangle
-\big|\mathring{S}^{[k]}(\omega)\big|^2
+(n-2k)\langle\mathring{S}^{[k]}(\omega),H\omega\rangle
+\langle\mathring{S}^{[k]}(\omega),
(\overline{\nabla}f)^{\bot}\omega\rangle\\
& \ \ \ \ +k(n-k)|H|^2|\omega|^2
+k\langle H,\overline{\nabla}f\rangle
|\omega|^2.
 \end{aligned}
\end{equation*}

From the proof of Theorem 1.1 in \cite{Cui-Sun} (see also Formula (18) of \cite{Raulot-Savo} and its proof), we have
\begin{equation}
\big|\mathring{S}^{[k]}(\omega)\big|^2
\leq\frac{k(n-k)}{n}|\mathring{B}|^2|\omega|^2.
\end{equation}
Therefore,
\begin{equation*}
\begin{aligned}
& \ \ \ \ \langle{\rm Ric}_{N,f}^{[k]}(\omega),\omega\rangle\\
&\geq k(n-k)\bigg(c+|H|^2-\frac{1}{n}\big|\mathring{B}\big|^2
-\frac{(n-2k)|H|+\big|(\overline{\nabla}f)^{\bot}\big|}
{\sqrt{nk(n-k)}}\big|\mathring{B}\big|
-\frac{\big|(\overline{\nabla}f)^{\bot}\big|}{n-k}|H|
\bigg)|\omega|^2.
 \end{aligned}
\end{equation*}
Let
\begin{equation*}
\begin{aligned}
\xi&=\underset{M}{\rm max}\big|(\overline
{\nabla}f)^{\bot}\big|,\\
\widetilde{\gamma}_{k}&=\min_{M}\Big\{|H|^2-\frac{1}{n}\big|\mathring{B}\big|^2
-\frac{(n-2k)|H|+\xi}{\sqrt{nk(n-k)}}\big|\mathring{B}\big|
-\frac{\xi}{n-k}|H|\Big\}.
\end{aligned}
\end{equation*}
Then
\begin{equation}\label{5.10}
\langle{\rm Ric}_{N,f}^{[k]}(\omega),\omega\rangle
\geq k(n-k)(c+\widetilde{\gamma}_{k})|\omega|^2.
\end{equation}
Substituting the curvature estimate \eqref{5.10} into \eqref{eqn:5.1}, we immediately obtain the estimate \eqref{1.14} for the first nonzero eigenvalue $\lambda_{1,k,f}$ of the weighted $p$-Laplacian $\Delta_{p,f}$
acting on $k$-forms of $M$. \qed

\end{document}